\documentclass{amsart}

\usepackage{MyStyleEN}
\usepackage{amssymb}
\usepackage[all]{xy}

\title[Pure subrings, endomorphism rings and Frobeniuses]{Pure subrings of regular local rings, endomorphism rings and Frobenius morphisms}
\author{Takehiko Yasuda}
\address{Department of Mathematics and Computer Science, 
Kagoshima University, 1-21-35 Korimoto, Kagoshima 890-0065, Japan}
\email{yasuda@sci.kagoshima-u.ac.jp}
\thanks{This research  was supported by  Grants-in-Aid for Scientific Research (22740020).}

\begin{document}

\maketitle

\begin{abstract}
The aim of this paper is threefold: first, to prove that the endomorphism ring
associated to a pure subring of a regular local ring is a noncommutative crepant resolution
if it is maximal Cohen-Macaulay; second, to see that in that situation, a different, but  Morita equivalent,
noncommutative crepant resolution can be constructed by using Frobenius morphisms;
finally, to study the relation
between Frobenius morphisms of noncommutative rings and the finiteness of global dimension. 
As a byproduct, we will obtain a result on wild quotient singularities: 
If the smooth cover of a wild quotient singularity
is unramified in codimension one, then the singularity is not strongly
F-regular.
\end{abstract}

%\tableofcontents

\section{Introduction}

\subsection{Pure subrings of regular local rings and NCCRs (Section \ref{sec-NCRs-pure})}

After Van den Bergh \cite{VandenBergh:2002tm}  introduced the notion of NCCR (noncommutative crepant resolution),
its importance is now well recognized. 
A nice reference of the theory is Leuschke's recent survey article \cite{Leuschke:2011te}.
The most basic example is the following: 

\begin{ex}\label{ex-intro-tame-quot}
Let $S:= k[[x_{1},\dots,x_{d}]]$, a power series ring over a field $k$, 
 $G \subset GL_{d}(k)$ a small finite subgroup with $ \mathrm{char}(k) \nmid \sharp G$,
 and $R := S^{G}$  the invariant subring.  Then 
the endomorphism ring   of the $R$-module $S$, $\End_{R}(S)$, is isomorphic to the skew group ring 
$S*G$ and so a NCCR. Namely it has finite global dimension and is a MCM (maximal Cohen-Macaulay) $R$-module.
(Unlike the original definition of NCCR by Van den Bergh, we do not assume that the base ring $R$ is Gorenstein.)
\end{ex}

We are interested in generalizing this example. In particular, we would like to investigate the following problem:

\begin{prob}\label{prob-intro-motivating}
Let $R \subset S$ be a module-finite  and pure extension of commutative domains such that they are both
Noetherian, complete, local and normal, and moreover $S$ is regular.
Then is $\End_{R}(S)$ a NCCR?
\end{prob}

Our main motivation for this problem is 
 the situation in 
positive characteristic. Then the extension $R \subset S$ can be purely inseparable
and we have no Galois group $G$ and cannot use the isomorphism $\End_{R}(S)\cong S*G$ 
to show that $\End_{R}(S)$ is a NCCR. 

The purity condition in the problem does not matter in characteristic zero, 
because every module-finite extension of normal domains is pure.
As Corollary \ref{cor-byproduct} shows, the condition is necessary in positive characteristic. 

Apart from Example \ref{ex-intro-tame-quot},  the answer to the problem is known to be positive in dimension two. 
Indeed, in this case, $\End_{R}(S)$
 is MCM, as every reflexive module is MCM.
 Moreover
$S$ contains every indecomposable MCM $R$-module as a direct summand 
\cite{Anonymous:xoYdZhOA}.
Then $\End_{R}(S)$ has finite global dimension from the following theorem:

\begin{thm}[\cite{Leuschke:2007va}]\label{thm-intro-Leuschke}
Let $S$ be a commutative Noetherian complete local CM ring
and $M$ a finitely generated MCM $S$-module
which includes every indecomposable MCM $S$-module as a direct summand.
Then $\End_{S}(M)$ has finite global dimension.
\end{thm}

In dimension $\ge 3$, 
 one can hardly expect that  a ring $R$ with a regular covering $S$ has only finitely many indecomposable MCM modules.
However we will generalize Theorem  \ref{thm-intro-Leuschke}    as follows: 
\begin{thm}[Theorem \ref{thm-pure-finite-gldim-main}]
Let $S$ and $M$ be as in Theorem \ref{thm-intro-Leuschke}. 
Suppose that $R \subset S$ is a pure subring such that $S$ is a finitely generated $R$-module.
Suppose also that $\Hom_{R}(S,M)$ is MCM. Then $\End_{R}(M)$ has finite global dimension.
\end{thm}

As a corollary we obtain a partial answer of Problem \ref{prob-intro-motivating}:

\begin{cor}[Corollary \ref{cor-pure-subring-regular-CMEnd->NCCR}]
Let $R$ and $S$ be as in Problem \ref{prob-intro-motivating}.
If $\End_{R}(S)$ is MCM, then
it is a NCCR. 
\end{cor}

Thus the  remaining problem is:

\begin{prob}
With the assumption as in Problem \ref{prob-intro-motivating},
 is $\End_{R}(S)$ always MCM? Or, when is it?
\end{prob}

\subsection{A result on wild quotient singularities (Section \ref{sec-byproduct})}

Using the above result and Yi's theorem of the global dimension of 
a skew group ring \cite{Yi:1994jl}, we will prove the following result on wild quotient singularities:

\begin{cor}[Corollary \ref{cor-byproduct}]
Let  $S:=k[[x_{1},\dots,x_{d}]]$, $G \subset \Aut_{k} (S)$  a finite group of order divisible by 
the characteristic of $k$ and $R := S^{G}$.  Suppose that $S$ is unramified over $R=S^{G}$ in codimension one.
Then $R$ is not a pure subring of $S$
or strongly F-regular.
\end{cor}

\begin{rem}
Several cases are known where the invariant ring $S^{G}$ is not Cohen-Macaulay,
hence nor F-regular. For instance, it is the case if $G \cong \ZZ/ p^{n} \ZZ$, the action is linear
and the fixed point locus has codimension $\ge 3$ \cite{Ellingsrud:1980wy}. 
\end{rem}

\subsection{NCCRs via Frobeniuses (Section \ref{sec-NCCRs-Frob})}

We are interested also in the role of Frobenius morphism in the theory of 
 noncommutative (crepant) resolution. (By a NCR (noncommutative resolution), we mean 
 an endomorphism ring $\End_{R}(M)$ having finite global dimension.)
Firstly an interesting problem similar to  Problem \ref{prob-intro-motivating} is:

\begin{prob}\label{prob-frob-NCCR}
Let $R$ be a commutative ring of characteristic $p>0$ such that its Frobenius map $R \hookrightarrow R^{1/p}$ is finite. Then one can consider the endomorphism ring $\End_{R}(R^{1/p^{e}})$ associated to the $e$-iterated Frobenius 
$R \hookrightarrow R^{1/p^{e}}$. When is it  a NC(C)R?
 \end{prob}
 
This is of particular intereset, because $\End_{R}(R^{1/p^{e}})\cong\End_{R^{p^{e}}}(R)$
and 
\[
D(R) := \bigcup_{e\ge 0} \End_{R^{p^{e}}}(R)
\]
 is the ring of differential operators on $R$.
 If for sufficiently large $e$'s, $\End_{R}(R^{1/p^{e}})$ have global dimensions bounded from above,
 then from \cite{Bervsteuin:1958wx}, $D(R)$ also has finite global dimension. 
The  ring $\End_{R}(R^{1/p^{e}})$ is also closely related to the F-blowup \cite{Toda:2009bw}.
 
The following is our partial answer to Problem \ref{prob-frob-NCCR}:
\begin{thm}[Corollary \ref{cor-generator-by-Frobenius}]
Let $S$ and $R$ be as in Problem \ref{prob-intro-motivating}.
Suppose that $\End_{R}(S)$ is MCM.
Then for sufficiently large $e$, $\End_{R}(R^{1/p^{e}})$ is a NCCR and Morita equivalent to $\End_{R}(S)$.
\end{thm}

In fact, the core of our proof is due to Hara \cite{Anonymous:xoYdZhOA}, who proved 
the corresponding result in dimension two.
Also 
Toda and the author \cite{Toda:2009bw} previously proved the theorem in the case of tame quotient singularities.
Their proof uses a result from the representation theory,
while Hara's argument replaces it with an elegant ring-theoretic argument.

We should notice that $\End_{R}(R^{1/p^{e}})$ is not  a NCCR for a general singularity $R$.
Indeed Dao proved \cite{Dao:2010vd} that if $R$ is a local complete intersection which is regular in codimension two,
then $\End_{R}(R^{1/p^{e}})$ is not a NCCR for any $e$. 

\subsection{Noncommutative Frobeniuses and global dimensions (Sections \ref{sec-NC-Koh-Lee} and \ref{sec-NC-Frob-gldim})}

For commutative rings or schemes, the Frobenius morphism has been exploited as
a tool to study singularities: Kunz's characterization of regularity (smoothness) \cite{Kunz:1969to}, the study of F-singularities (see \cite{Anonymous:OoDKDTwF}) among others.
In \cite{Yasuda:2009wn}, the author defined 
the Frobenius morphism of the endomorphism ring of a module over a commutative ring. 
In Section \ref{sec-NC-Frob-gldim},
we will also define the Frobenius morphism of the skew group ring associated to 
a commutative ring and a finite group in an obvious way.
Then we will study relation between the finiteness of global dimension and the flatness of Frobenius morphism.
For commutative rings,  both properties are characterizations of regularity and
 Herzog \cite{Herzog:1974vk} gave a direct proof that the latter implies the former.
Koh and Lee  \cite{Koh:1998cp} refined his result and proved certain constraint which the minimal resolution
of every module satisfies. We will prove the noncommutative version of Koh and Lee's result (Theorem \ref{thm-koh-lee}).

However, in order to obtain finite global dimension using this, the flatness of Frobenius morphism
is not sufficient unlike the commutative case.
Then we will axiomize the properties of Frobenius morphism which are necessary to 
deduce finite global dimension. Moreover if the relevant noncommutative ring is $\End_{R}(M)$,
then we will relate the properties with those of $M$ in terms of the Frobenius of $R$.

\subsection{Convention}\label{subsec-convention}

Throughout the paper, $R$ denotes some commutative Noetherian local complete  domain of Krull dimension $d$.
Unless otherwise noted, a \emph{ring} means a (commutative or noncommutative) $R$-algebra which is a finitely generated torsion-free
$R$-module. Thus every commutative ring has Krull dimension $d$.

A \emph{module} means a finitely generated left module unless otherwise noted. 
Then the category of modules over a ring has the Krull-Schmidt property: Every module uniquely decomposes 
into the direct sum of indecomposable modules.
A ring or module is called \emph{maximal Cohen-Macaulay} or simply \emph{Cohen-Macaulay} (for short, CM)
 if they are maximal Cohen-Macaulay $R$-modules. Similarly a ring or module is called \emph{reflexive} if it is a reflexive $R$-module. The notation ${}_{R}M$  (resp.\ $M_{R}$) means that $M$ is a left (resp.\ right) $R$-module.

We denote the category of $R$-modules by $\mods(R)$
and subcategories of projective (resp.\ CM, reflexive) $R$-modules by 
$\proj(R)$ (resp.\ $\CM(R)$, $\refls(R)$).
A sequence of modules in such a subcategory is said to be \emph{exact} if
it is exact in the ambient category $\mods(R)$. 
A functor between such subcategories is said to be \emph{exact} if it preserves exact sequences.

\subsection{Acknowledgments}

I would like to thank Nobuo Hara, Colin Ingalls, Osamu  Iyama, Hiroyuki Nakaoka, Tadakazu Sawada
 and Shunsuke Takagi for useful discussions.

\section{NCRs of  pure subrings}\label{sec-NCRs-pure}

\subsection{Preliminaries}

First of all we recall some basic notions.

\begin{defn}
For a module $M$ over a ring $S$, the \emph{additive closure} of $M$, denoted $\add(M)$, is the category
consisting of modules isomorphic to a direct summand of $M^{\oplus l}$, $l \ge 0$.
We say that $M$ is an \emph{additive generator}  of the category $\add(M)$.
\end{defn}

\begin{defn}
We say that  a ring $S$ is \emph{of finite CM type}
if there are, up to isomorphism, only finitely many indecomposable CM $S$-modules.
Then there is an additive generator $M$ of $\CM(S)$, which is a CM module containing every indecomposable CM module as a direct summand.
 Such $M$ is called a \emph{CM generator} (over $S$).
\end{defn}

\begin{defn}
A (necessarily module-finite) extension $R\subset S$ of commutative rings is said to be \emph{pure} if the inclusion map $R\hookrightarrow S$ splits as an $R$-module map.
If it is the case, we also say that $R$ is a \emph{pure subring} of $S$.
\end{defn}

\begin{defn}
Let $\Lambda$ be a ring and $M$ a $\Lambda$-module.
The \emph{reflexive hull} of  $M$ is defined as the $\Lambda$-module
\[
M^{\refl} :=\Hom_{R}(\Hom_{R}(M,R),R) .
\] 
Then $M$ is reflexive if and only if the natural map $M \to M^{\refl}$ is an isomorphism.
\end{defn}

%\begin{rem}
%If $S$ is a normal ring, the above reflexivity for $S$-modules is equivalent to the usual one.
%Namely an $S$-module $M$ is reflexive if and only if $M \to \Hom_{S}(\Hom_{S}(M,S),S)$. 
%\end{rem}

\begin{defn}[\cite{Bondal:2002ue}, \cite{VandenBergh:2002tm}]
Let $M$ be a tosion-free $R$-module.
\begin{enumerate}
\item $\End_{R}(M)$
is called a \emph{NCR} if it has finite global dimension.
\item  $\End_{R}(M)$
is called a \emph{NCCR} if it is CM and has finite global dimension.
\end{enumerate}
\end{defn}

\subsection{NCRs of  pure subrings}

\begin{thm}\label{thm-pure-finite-gldim-main}
Let $S$ be a commutative CM ring of finite CM type with a CM generator $M$. 
Suppose that $R\subset S$ is pure and 
 that $\Hom_{R} (S,M)$ is CM. 
Then 
\[
\gldim \End_{R}(M)\le \max \{d,2\}.
\]
In particular, $\End_{R}(M)$ is a NCR.
Moreover if $d \ge 2$, then 
\[
\gldim \End_{R}(M)=d.
\]
\end{thm}

\begin{proof}
Let $\Lambda $ be the opposite ring of $\End_{R}(M)$.
Being  Noetherian, $\Lambda$ and $\End_{R}(M)$ 
have equal global dimension. 
Therefore we may show the theorem for  $\Lambda$ instead of $\End_{R}(M)$.

Notice that since $S$ is CM, we have $_{S}S \in \CM(S)=\add(_{S}M)$ and $_{R}S \in \add( _{R}M)$.
Since $_{R}R$ is a direct summand of $_{R}S$, we have $_{R}R \in \add(_{R}M)$.
Then we define functors
\[
\xymatrix{
\add(_{R}M)  \ar@/^1pc/[rrr]^{\alpha:=\Hom_{R}(M,-)} && &
\proj (\Lambda) \ar@/^1pc/[lll]^{\alpha^{-1} := \Hom_{\Lambda}( \alpha(R) , - ) }
}
\]
These are equivalences which are inverses to each other. (The equivalence is known as the projectivization \cite{Auslander:1997uh}.)
Let $e:=\max \{d,2\}$, $A$  an arbitrary $\Lambda$-module and 
\[
P_{\bullet} : P_{e-1} \to \cdots \to P_{1} \to P_{0}
\]
 the first $e$-step of a projective resolution of $A$. 
Set 
\[
L_{\bullet}:= \alpha^{-1}(P_{\bullet}) : L_{e-1} \to \cdots \to L_{1} \to L_{0}
\]
 and $L_{e} := \Ker (L_{e-1} \to L_{e-2})$. 

\begin{claim}\label{claim}
We have that $L_{e} \in \add(_{R}M)$.
\end{claim}

If the claim is true, then applying $\alpha$ to
\[
0 \to L_{e} \to L_{e-1} \to \cdots \to L_{0},
\]
we obtain an exact sequence 
\begin{equation}\label{projective resolution of A}
0 \to \alpha(L_{e}) \to P_{e-1} \to P_{e-2} \to \cdots \to P_{0}. 
\end{equation}
Here  the exactness follows from the left exactness of 
$\alpha$ and the exactness of $P_{\bullet}$. Sequence \eqref{projective resolution of A} is a projective resolution of $A$. Hence $\Lambda$ has global dimension $\le e$.  
On the other hand, from \cite{Ramras:1969tx} (see also \cite{Leuschke:2011te}), it has global dimension $\ge d$. Hence if $d\ge 2$,
then the equality in the theorem holds.

The proof up to this point is basically the same as the one in \cite{Leuschke:2007va}.
The difference lies in the proof of Claim \ref{claim}.

\begin{proof}[Proof of  Claim \ref{claim}]

By assumption,  $\Hom_{R}(S,N)$ is CM for every $N \in \add(_{R}M)$.
Hence 
\[
{}_{S}\Hom_{R}(S,N)\in\add(_{S}M) = \CM(S).
\]
Thus we have the functor
\[
\psi := \Hom_{R}(S,-) : \add(_{R}M) \to \CM(S).
\]

  Set $\Gamma:=\End_{S}(M)^{\op}$ and
define equivalences
\[
\xymatrix{
\CM(S) \ar@/^1pc/[rrr]^{\beta:=\Hom_{S}(M,-) } && &
\proj (\Gamma)  \ar@/^1pc/[lll]^{\beta^{-1} := \Hom_{\Gamma}( \beta(S) , - ) } 
},
\]
which are inverses to each other.
Since $\Gamma$ is a subring of $\Lambda$, we have the forgetting functor
\[
 \phi :\mods(\Lambda) \to \mods (\Gamma), \ {}_{\Lambda}A \mapsto {}_{\Gamma}A ,
\]
which is obviously exact.
For  $N \in  \add(_{R}M)$, we 
have isomorphisms of  $\Gamma$-modules, 
\begin{align*}
  \phi \circ  \alpha(N) 
  & = \Hom_{R}(M,N) \\
 & \cong \Hom_{R}(S\otimes_{S}M,N) \\
 &  \cong \Hom_{S}(M,\Hom_{R}(S,N) )  \\
&   \cong \beta \circ \psi(N).
\end{align*}
Therefore the following diagram is (2-)commutative:  
\[
\xymatrix{
\proj (\Lambda) \ar[r]^{\phi|_{\proj(\Lambda)}} & \proj (\Gamma) \\
\add(_{R}M) \ar[u]^{\alpha} \ar[r]_{\psi} & \CM(S)\ar[u]_{\beta}
}
\]

Put $K_{\bullet} := \psi(L_{\bullet})$.  
Since 
\[
K_{\bullet} \cong \beta^{-1} (_{\Gamma}P_{\bullet}),
\]
and $\beta^{-1}$ is exact,
$K_{\bullet} $ is exact.  Let 
\[
K_{e} := \Ker (K_{e-1} \to K_{e-2}) \cong \psi(L_{e}) . 
\]
From the depth lemma, $K_{e}$ is CM 
and  belongs to $\add(_{S}M)=\CM(S)$.
Hence $_{R}K_{e} \in \add(_{R}M)$.
Since $R\subset S$ is pure, 
 $L_{e} \cong \Hom_{R}(R,L_{e})$ is a direct summand of $_{R}K_{e} \cong \Hom_{R}(S , L_{e})$.
 Hence the claim holds. 
\end{proof} 
We have completed the proof of Theorem \ref{thm-pure-finite-gldim-main}.
\end{proof}

\begin{rem}\label{rem-replace-Hom-tensor}
From Lemma \ref{lem-Hom-tensor-MCM}, if $S$ is normal, then we can replace
the condition in Theorem \ref{thm-pure-finite-gldim-main} that $\Hom_{R}(S,M)$ is CM
 with the condition that $(S\otimes_{R}M)^{\refl}$ is CM.
\end{rem}

\begin{lem}\label{lem-Hom-tensor-MCM}
Let $S$ be a commutative normal CM ring of finite CM type with a CM generator $M$.
Then  for any $R$-module $N$, 
$(N\otimes_{R}M)^{\refl}$ is MCM if and only if so is $\Hom_{R}(N,M)$.
\end{lem}

\begin{proof}
Let $K_{S}$ denote the canonical module of $S$.
We have
\begin{align*}
\Hom_{S}((N\otimes_{R}M)^{\refl}, K_{S})\cong \Hom_{S}(N\otimes_{R}M, K_{S}) 
\cong\Hom_{R}(N,  \Hom_{S}(M,K_{S}) ) . 
\end{align*}
Here the left isomorphism follows from the fact $K_{S}$ is reflexive.
Thus $(N \otimes_{R}M)^{\refl}$ and $\Hom_{R}(N, \Hom_{S}(M,K_{S}))$
are reflexive modules which are  the canonical duals to each other. 
Since $\Hom_{S}(-,K_{S})$ is an auto-equivalence of $\CM(S)$ (see \cite{Yoshino:1990ed}),
 $\Hom_{S}(M,K_{S})$ is also a CM generator over $S$. 
Therefore $\Hom_{R}(N,M)$ is CM if and only if $\Hom_{R}(N,  \Hom_{S}(M,K_{S}) )$ is CM
if and only if 
$(N \otimes_{R}M)^{\refl}$ is CM.
\end{proof}

\begin{ex}
Let $S$ be a commutative normal
CM ring of finite CM type with a CM generator $M$,
  $G$ a finite group acting on $S$ and $R:=S^{G}$. Suppose  that 
the ring extension $R \subset S$ is pure and unramified in codimension one.
Then $S*G \cong \End_{R}(S)$ (see \cite[page 118]{Auslander:1962wi}). 
Hence we have isomorphisms of $S$-modules, 
\[
M^{\oplus \sharp G} \cong (S*G) \otimes _{S}M\cong \End_{R} (S) \otimes_{S}M\cong  \Hom_{R}(S,M).
\]
(Here the last isomorphism holds, since the both hand sides are reflexive modules
and the isomorphism is valid in codimension one.)
In particular, $\Hom_{R}(S,M)$ is CM. From Theorem \ref{thm-pure-finite-gldim-main}, $\End_{R}(M)$
has global dimension $d$ and is a NCR.
\end{ex}

\begin{cor}\label{cor-pure-subring-regular-CMEnd->NCCR}
Let $S$ be a commutative regular local ring with the extension $R\subset S$ pure.
Then if $\End_{R}(S)$ is CM, then it is a NCCR.
\end{cor}

\begin{proof}
Since $S$ is regular, $_{S}S$ is a CM generator. Now the corollary is a direct consequence of Theorem \ref{thm-pure-finite-gldim-main}.  
\end{proof}

\section{A result on wild quotient singularities}\label{sec-byproduct}

In this section, we suppose that the base ring $R$ has characteristic $p>0$. 
We will always denote by $q$ a power of $p$; $q=p^{e}$ ($e\in \Znn$).
We define $R^{1/q} := \{ f^{1/q} | f \in R \}$ in the algebraic closure of the quotient field of $R$. 
Then $R$ is a subring of $R^{1/q}$ and $R^{1/q}$ has a natural $R$-module structure.
The Frobenius map, $F:R \to R$, $f \mapsto f^{p}$, is isomorphic to the inclusion map 
$R \hookrightarrow R^{1/p}$.
We will make also the assumption that $R$ is F-finite. Namely the $R$-module $_{R}R^{1/p}$ is finitely generated.

\begin{defn}[\cite{Hochster:1989tk}]
We say that $R$ is \emph{strongly F-regular} if for every  $0 \ne c \in R$, 
there exists $q=p^{e}$ such that the $R$-linear map
 $R \to {}_{R}R^{1/q}$ with $1 \mapsto c^{1/q}$
splits.
\end{defn}

\begin{defn}
Let $S$ be a ring and $G$ a  group acting on it.
Then  the \emph{skew group ring} $S*G$ is defined as follows: It is a free $S$-module, $\bigoplus_{g\in G} S \cdot g$, 
with the multiplication defined by $(s g)(s'g') =(s g(s') )(gg')$. 
\end{defn}

\begin{cor}\label{cor-byproduct}
Let $S$ be a commutative regular ring  and 
$G\subset \Aut(S)$ a finite group of automorphisms of $S$. 
Suppose that the induced $G$-action on the residue field of $S$ is trivial 
and that $S$ is unramified over $R:=S^{G}$ in codimension one.
Let $\Lambda := S*G $, which is by assumption isomorphic to $\End_{R}(S)$.
Then the following are equivalent:
\begin{enumerate}
\item Either $p = 0$ or $p \nmid \sharp G$.
\item $\Lambda$ has finite global dimension.
\item $\Lambda$ has global dimension $d$.
\item The ring extension $R \subset S$ is pure.
\item $R$ is strongly F-regular.
\end{enumerate}
\end{cor}

\begin{proof}

(1)$\Leftrightarrow$(2): \cite[Theorem 5.2]{Yi:1994jl}.

(1)$\Rightarrow$(3): \cite[7.5.6]{McConnell:2001vd}.

(3)$\Rightarrow$(2): Obvious.

(1)$\Rightarrow$(4): Well-known.

(4)$\Rightarrow$(1): This follows from
 Corollary \ref{cor-pure-subring-regular-CMEnd->NCCR}.  

(4)$\Rightarrow$(5): \cite{Hochster:1989tk}.

(5)$\Rightarrow$(4): A strongly F-regular ring is a splinter, that is, every module-finite extension of it is pure 
\cite{Huneke:1996tr}. 
\end{proof}

See Corollary \ref{cor-skew-seven-equivalent} for two more equivalent conditions.

\section{NCCRs via Frobenius}\label{sec-NCCRs-Frob}

In this section, we will continue to suppose that the base ring $R$ has characteristic $p>0$.

The following is a straightforward generalization of Hara's similar result in dimension two \cite{Anonymous:xoYdZhOA}.

\begin{prop}\label{prop-Hara}
Let $S$ be a commutative CM ring of finite CM type with a CM generator $M$.
We suppose that $\End_{R}(M)$ is CM (see Remark \ref{rem-replace-Hom-tensor}),  that $R \subset S$ is pure and that $R$ is strongly F-regular.
Then for sufficiently large $e$, $F^{e}_{*}R = {}_{R}(R^{1/p^{e}})$ is an additive generator of $\add(_{R}M)$. 
\end{prop}

\begin{proof}
Our proof is essentially the same as Hara's one. 
Firstly, for every $q$, since $_{S}S^{1/q} \in\CM(S)= \add(_{S}M)$, we have $_{R}S^{1/q} \in \add(_{R}M)$.
Being a direct summand of $_{R}S^{1/q}$, $_{R}R^{1/q}$ is also in $\add(_{R}M)$.

There exists $0 \ne c \in R$ such that for every $N \in \add(_{R}M)$,
the map  
$c: N \to N$, $m \mapsto cm$ factors through a free $R$-module as an $R$-linear map:
\[
c : N \to R^{\oplus \rank N} \to N. 
\]
Indeed every indecomposable $L \in \add(_{R}M)$ is embedded in the free module of the same rank, 
as it is torsion-free. Then for sufficiently factorial $c_{L} \in R$, the image of the induced map 
$c_{L}:R^{\oplus \rank L} \to R^{\oplus \rank L}$ is contained in $L$. We can choose $\prod_{L}c_{L}$
as the desired $c$.

If we put $M^{*} := \Hom_{R}(M,R)$, it is by assumption CM.
Hence  $_{S}M^{*}\in \add(_{S}M)$ and  $_{R}M^{*} \in \add(_{R}M)$. 
Let $M^{1/q}$ be the $S^{1/q}$-module
corresponding to $M$ by the obvious isomorphism $S \cong S^{1/q}$.
Since $\Hom_{R}(M^{*}, S^{1/q})$ is CM, $_{S^{1/q}}\Hom_{R}(M^{*}, S^{1/q}) \in \add(_{S^{1/q}}M^{1/q})$.
Therefore  $_{R^{1/q}}\Hom_{R}(M^{*},R^{1/q})\in \add(_{R^{1/q}}M^{1/q})$.
Hence the $R^{1/q}$-linear map
\[
c^{1/q}  :  \Hom_{R}(M^{*},R ^{1/q}) \to \Hom_{R}(M^{*},R ^{1/q})
\]
factors as 
\[
 \Hom_{R}(M^{*},R ^{1/q}) \to (R^{1/q})^{\oplus m} \to  \Hom_{R}(M^{*},R ^{1/q}) ,
\]
where $m$ is the rank of $M$.

On the other hand,
since $R$ is strongly F-regular, there exists $q=p^{e}$ such that the $R$-linear map
$R \to R^{1/q}$, $1 \mapsto c^{1/q}$ splits.
Applying $\Hom_{R}(M^{*},-)$ to it, we obtain a splitting $R$-linear map
\begin{equation}\label{split-hom}
M=\Hom_{R}(M^{*},R) \to \Hom_{R}(M^{*},R ^{1/q}).
\end{equation}
This factors as
\[
M \hookrightarrow  \Hom_{R}(M^{*},R ^{1/q}) \xrightarrow{c^{1/q}}  \Hom_{R}(M^{*},R ^{1/q}) ,
\]
and furthermore as
\[
M \hookrightarrow  \Hom_{R}(M^{*},R ^{1/q}) \to (R^{1/q})^{\oplus m} \to  \Hom_{R}(M^{*},R ^{1/q}) . 
\]
The splitting of \eqref{split-hom} yields that of $M \to (R^{1/q})^{\oplus m}$.
Hence $_{R}M$ is a direct summand of $_{R}(R^{1/q})^{\oplus m}$.
In consequence, $_{R}R^{1/q}$ is an additive generator of $\add(_{R}M)$.
\end{proof}

\begin{cor}\label{cor-generator-by-Frobenius}
With the assumption as in Proposition \ref{prop-Hara},
for sufficiently large $e$, $\End_{R}(R^{1/p^{e}}) $ is Morita equivalent to $\End_{R}(M)$
and a NCCR as well.
\end{cor}

\begin{proof}
As is well-known, the equality $\add(_{R}R^{1/q})=\add(_{R}M)$ induces the Morita equivalence of
$\End_{R}(R^{1/q})$ and $\End_{R}(M)$.
\end{proof}

The ring of differential operators on $R$ is expressed as follows \cite{Yekutieli:1992th}:
\[
D(R) = \bigcup_{e\ge 0} \End_{R^{p^{e}}}(R).
\]

\begin{cor}
With the assumption as  in Proposition \ref{prop-Hara}, 
$D(R)$ has global dimension $\le d+1$. 
\end{cor}

\begin{proof}
We have obvious isomorphisms $\End_{R^{p^{e}}}(R) \cong \End_{R}(R^{1/p^{e}})$.
Hence $\End_{R^{p^{e}}}(R)$ has global dimension $d$ for $e \gg 0$.
Since $D(R)$ is a direct limit of them, the corollary follows from a general result on the global dimension of
direct limits by Ber{\v{s}}te{\u\i}n \cite{Bervsteuin:1958wx}.
\end{proof}

\section{Noncommutative Herzog-Koh-Lee and exact order-raising endofunctors}\label{sec-NC-Koh-Lee}

\subsection{Noncommutative Herzog-Koh-Lee}

Koh and Lee \cite{Koh:1998cp} proved a result 
 on the minimal resolution of a module over a commutative local ring, 
which is a refinement of Herzog's result \cite{Herzog:1974vk}. 
We will generalize their result to the noncommutative setting
along the lines of  Koh and Lee.

Let $\Lambda$ be a ring.
From assumption (see \S\ref{subsec-convention}) and \cite[(23.3)]{Lam:2001ut},
 $\Lambda$ is semiperfect and every (left or right) finitely generated $\Lambda$-module admits the minimal projective resolution.
Let $e_1, \dots, e_l$ be  a basic set of  primitive idempotents for $\Lambda$ so that
 $Q_i:= \Lambda e_i$, $i=1,\dots,l$, are the irredundant set of 
 the indecomposable projective left $\Lambda$-modules (see \cite[Proposition 27.10]{Anderson:1992vt}).

Let $\fj \subset \Lambda$ be the Jacobson radical.
The \emph{socle} of a right $\Lambda$-module $V$ is defined to be its largest semisimple submodule and denoted by $\Soc (V)$.  Since $\Lambda$ is semiperfect, $\Soc (V)$ is 
equal to the annihilator of $\fj$ (see \cite[pages 118 and 171]{Anderson:1992vt}) :
\[
\Soc (V) =\{v \in V | v \fj =0\} \subset V.
\]

Let $\fm$ denote the maximal ideal of $R$.
Let $x_{1}, \dots, x_{r} \in \fm$ be a maximal $\Lambda$-sequence,
where $r$ is the depth of $_{R}\Lambda$  and we have $r \le d$.
 Define a right $\Lambda$-module $V := \Lambda/ \sum_{i} x_{i}\Lambda$. 
Then  $V_{R}$ has depth zero
and nonzero socle. 
Since for some $n$, we have $\fj^{n} \subset \fm \Lambda$  \cite[(20.6)]{Lam:2001ut}, $V_{\Lambda}$ also has nonzero socle. 
Indeed if $m$ is the smallest number such that $ \Soc (V_{R}) \cdot \fj^{m} = 0$,
then 
\[
0 \ne \Soc (V_{R}) \cdot \fj^{m-1} \subset \Soc (V_{\Lambda}). 
\]
From now on, we simply write $\Soc(V) = \Soc (V_{\Lambda})$.
For some $ 1 \le i \le l $,  we have $ \Soc(V) \otimes _{\Lambda}Q_{i} \ne 0 $, say 
$ \Soc(V) \otimes _{\Lambda}Q_{1} \ne 0 $. Then set 
\[
 s_{\Lambda} := \inf \{ t \ge 1 |  \Soc (V)\otimes _\Lambda Q_1 \not \subset V \fm^{t}  \otimes_\Lambda Q_1\} .
\]
Let
\[
P_\bullet :\cdots \to P_{j+1} \xrightarrow{\delta_{j+1}} P_j \xrightarrow{\delta_j} P_{j-1} \to \cdots \to P_0 
\]
be the minimal projective resolution of a left 
$\Lambda$-module $U$. Here for each $j$, we can write
$P_j = \bigoplus_{\lambda=1}^{l} Q_i ^{\oplus n_{j,i}}$.

\begin{lem}\label{lem-Koh-Lee-NC}
For $j >r$, in particular, for $j >d$, we have either that $n_{j,1}=0$ or that
 $\delta_{j+1}$ is nonzero modulo $ \fm^{s_{\Lambda}} $ (that is, $\Image \delta_{j+1} \not \subset \fm^{s_{\Lambda}} P_{j}$). 
\end{lem}

\begin{proof}
Let $j >r$.
The right $\Lambda$-module $V$ has projective dimension $r$ \cite[(5.32)]{Lam:1999tn}.
Hence $\Tor_{j}^{\Lambda}(V,U)=0$. 
Hence
\[
V \otimes_\Lambda P_\bullet  : \cdots \to V \otimes_\Lambda P_{j+1} \xrightarrow{\bar \delta_{j+1}}
 V \otimes_\Lambda P_j\xrightarrow{\bar \delta_j} V \otimes_\Lambda P_{j-1} \to \cdots 
\]
is exact in the middle. Since $P_\bullet$ is minimal, we have 
\[
\Soc (V) \otimes_\Lambda P_{j} \subset \Ker \bar  \delta_j =\Image \bar \delta_{j+1}.
\] 
To obtain a contradiction, we make the assumptions that $n_{j,1}>0$ and that $\delta_{j+1}$ were zero modulo $ \fm^{s_{\Lambda}} $. From the latter, 
\[
\Soc (V) \otimes_\Lambda P_{j} \subset \Image \bar \delta_{j+1} \subset    V\fm^{s_{\Lambda}} \otimes_\Lambda P_j.
\] 
From the former, this implies that  $\Soc (V) \otimes_\Lambda Q_{1} \subset    V\fm^{s_{\Lambda}} \otimes_\Lambda Q_{1} $, which contradicts the definition of $s_{\Lambda}$. We have proved the proposition. 
\end{proof}

For a nonempty subset $I \subset \{1,2, \dots, l\}$, we define a ring 
$\Lambda_{I} := (\sum_{ i \in I } e_{i})\Lambda (\sum_{ i \in I } e_{i})$ and a number $s_{I}=s_{\Lambda_{I}}$ in the same way
as $s_{\Lambda}$. Then we put 
\[
s := \max \{s_{I} |  \emptyset \ne I \subset \{1,2, \dots, l\} \}.
\]

\begin{thm}[Noncommutative Herzog-Koh-Lee]\label{thm-koh-lee}
Suppose that for every $j  \le l (d+1) $, $\delta_{j+1}$ is zero modulo $ \fm^{s} $.
Then $P_{l(d+1)}=0$. Equivalently $\projdim U <  l(d+1) $. 
\end{thm}

\begin{proof}
The proof is by induction on $l$.
If $l =1$, then the proposition is a direct consequence of  Lemma \ref{lem-Koh-Lee-NC}.  We now turn to the general case.
For  $d<  j \le l(d+1)  $, since $\delta_{j+1}$ is zero modulo $ \fm^{s} $, 
again from Lemma \ref{lem-Koh-Lee-NC},
we have $n_{j,1}=0$. 
Put $e:= \sum_{i \ge 2} e_{i}$ and $\Lambda' :=e\Lambda e$.
Then consider the complex $P'_{\bullet} := (eP_{\bullet})[d+1]$ of $B$-modules.  
Namely the complex $P'_{\bullet}$ is defined by $P'_{i} := eP_{i+d+1}$ with the obvious differentials.
For $-1 <   j \le (l-1)(d+1)  $, since $n_{j+d+1 , 1}=0$, $P'_{j}=eP_{j+d+1}$
is a projective $\Lambda'$-module. 
Hence  $ P'_{\bullet}$
is the minimal projective resolution of $\Coker(P'_{1}\to P'_{0})$ at least
 in degree $\le (l-1)(d+1)$ such that for $  j \le (l-1)(d+1)  $, the differential 
 $\delta'_{j+1} : P'_{j+1} \to P'_{j}$ is zero modulo $ \fm^{s} $. 
From the induction hypothesis,    we have 
\[
  eP_{l(d+1)} =  P'_{(l-1)(d+1)}  =0 \text{ and } P_{l(d+1)}=0. 
 \]
\end{proof}

When $\Lambda$ is CM, we do not need the inductive argument and have a better result. In this case,  we have
$r := \depth (_{R}\Lambda) = d$
and for every $i$, $\Soc(V) \otimes_{\Lambda} Q_{i}\ne 0$. 
Set 
\[
 s'_{\Lambda} :=  \max_{i} \{ \inf \{ t \ge 1 |  \Soc (V)\otimes _\Lambda Q_{i} \nsubseteq V \fm^{t}  \otimes_\Lambda Q_{i} \} \} .
\]
Then we can see the following by an argument similar to that of Lemma \ref{lem-Koh-Lee-NC}.

\begin{thm}\label{thm-koh-lee-CM}
Suppose that $\Lambda$ is CM.
For $j >d$, if $\delta_{j+1}$ is zero modulo $ \fm^{s'_{\Lambda}} $,
then $P_{j}=0$ and $\projdim U \le d$.
\end{thm}

\subsection{Exact order-raising endofunctors and the finiteness of global dimension}

We  keep the notation of the preceding subsection. 

\begin{defn}\label{defn-order-raising}
Let $\Phi : \proj(\Lambda) \to \proj(\Lambda)$ be an endofunctor.
We say that $\Phi$ is \emph{order-raising} if for every $i>0$, there exists  $e_{0}>0$ such that for every $e\ge e_{0}$ and
for every morphism $f:P \to Q$ in $\proj(\Lambda)$ which factors through $\fj Q$, $\Phi^{e}(f) : \Phi^{e}(P) \to \Phi^{e}(Q)$ factors through 
$\fj^{i} \Phi(Q)$. 
\end{defn}

\begin{defn}
A functor $\Phi$ between subcategories of abelian categories 
is said to \emph{have zero kernel} if $\Phi(N)\cong 0 \Rightarrow N \cong 0$.
\end{defn}

\begin{cor}\label{cor-endofunctor-NCCR}
Suppose that there exists an exact and order-raising endofunctor $\Phi : \proj(\Lambda) \to \proj(\Lambda)$
which has zero kernel.
Then $\Lambda$ has finite global dimension. Moreover if $\Lambda$ is CM, then $\Lambda$ has
global dimension $d$.
\end{cor} 

\begin{proof}
Since for some $n$, $\fj^{n} \subset \fm \Lambda$, we may replace $\fj^{i} \Phi(Q)$ in Definition \ref{defn-order-raising}
with $\fm^{i} \Phi(Q)$.
 Let $U$ be an arbitrary finitely generated $\Lambda$-module
and 
\[
P_\bullet :\cdots \to P_{j+1} \to P_j \to P_{j-1} \to \cdots \to P_0 
\]
its minimal projective resolution. 
Since $\Phi$ is exact,  for every $e$,
$\Phi^{e}(P_{\bullet})$ is an exact sequence.
Since $\Phi$ is order-raising, if $e$ is sufficiently large, then
\[
\Phi^{e}(P_{l(d+1) }) \to \Phi^{e}(P_{l(d+1)-1}) \to \cdots \to \Phi^{e}(P_{0})
\]
is the first steps of the minimal projective resolution of $\Phi^{e}(U)$, 
whose differentials are zero  modulo $\fm^{s}$.  
From Theorem \ref{thm-koh-lee}, $\Phi^{e}(P_{l(d+1) })=0$.
Since $\Phi$ has zero kernel, $P_{l(d+1)} = 0$.
Therefore $\projdim U \le l(d+1)$ and hence $\gldim \Lambda< \infty$.

For the second assertion, using Theorem \ref{thm-koh-lee-CM} instead, we can similarly show that $\gldim \Lambda \le d$.
On the other hand, from \cite{Ramras:1969tx} (see also \cite{Leuschke:2011te}), $\gldim \Lambda \ge d$,
and the corollary follows.
\end{proof}

\section{Noncommutative Frobeniuses and the finiteness of global dimension}\label{sec-NC-Frob-gldim}

In this section we will define the Frobenius morphism for two classes of noncommutative rings,
endomorphism rings of  modules  and skew group rings. Then we study when the Frobenius pullback functor
for such a ring satisfies the conditions in the last section
for that the ring has finite global dimension. 

We now suppose that $R$ is normal and of characteristic $p>0$.

\subsection{Frobeniuses of endomorphism rings}

Let $M$ be a nonzero reflexive $R$-module and $M^{1/p}$ the corresponding $R^{1/p}$-module
under the isomorphism $R^{1/p}\cong R$, $f^{1/p} \leftrightarrow f$.  Put $E := \End_{R}(M)$ and $E^{1/p} := \End_{R^{1/p}}(M^{1/p})$.
The Hom set
\[
H:= \Hom_{R}(M,M^{1/p})
 \]
  has a natural $(E^{1/p},E)$-bimodule structure. 

\begin{defn}[\cite{Yasuda:2009wn}]\label{defn-Frobenius-for-end}
We define the \emph{Frobenius 
pullback} of $E$ as
\[
\bF^{*}:= H \otimes_{E}- : \mods(E) \to \mods(E^{1/p}),
\]
	and the \emph{Frobenius pushforward} as its right adjoint
\[
\bF_{*}:= \Hom_{E^{1/p}}( H, -  )  :\mods(E^{1/p})\to \mods(E) .
\]
We call the pair  $(\bF^{*},\bF_{*})$ the \emph{Frobenius morphism} of $E$.
\end{defn}

Note that $E$ and $E^{1/p}$ are canonically isomorphic and one may regard the above functors 
as endofunctors.

\begin{ex}
If $M=R$, then $H= R^{1/p}$, $E=R$ and $E^{1/p} = R^{1/p}$.
Therefore $\bF^{*}$ and $\bF_{*}$ are respectively the pullback and pushforward of the ordinary
relative Frobenius $R \hookrightarrow R^{1/p}$.
\end{ex}

\begin{defn}
Let $F:R \to R$, $f\mapsto f^{p}$ be the absolute Frobenius map of $R$.
We define the \emph{reflexive pullback} of an $R$-module $N$ by $F$ as
 $F^{\times} N := (F^{*} N)^{\refl}$ .
\end{defn} 
 
\begin{defn} 
Given an $R$-module $N$, we define as follows:
\begin{enumerate}
\item  $N$ is  \emph{$F^{\times}$-closed}   if $F^{\times}N \in \add(N)$.
\item $N$ is \emph{$F_{*}$-closed} if $F_{*}N \in \add(N)$.
\item $N$ is \emph{strongly $F_{*}$-closed} 
if $F_{*}N$ is an additive generator of $\add(N)$.
\end{enumerate}
\end{defn}

\begin{lem}\label{lem-proj-right-left}
Let $N$ be a reflexive $R$-module. 
Consider the following conditions:
\begin{enumerate}
\item $N \in \add(M)$.
\item $\Hom_{R}(M,N)$ is a projective right $E$-module.
\item $\Hom_{R}(N,M)$ is a projective left $E$-module.
\end{enumerate}
Then $(1) \Leftrightarrow (2) \Rightarrow (3)$. 
\end{lem}

\begin{proof}
We have an equivalence of categories of reflexive modules
\[
\Hom_{R}(M,-) : \refls(R) \xrightarrow{\sim} \refls(E^{\op}),
\] which restricts to 
an equivalence $\add(M) \xrightarrow{\sim} \proj(E^{\op})$ (see \cite{Iyama:2008io}). This shows $(1) \Leftrightarrow (2) $.

If  $N \in \add(M)$, then $_{E}\Hom_{R}(N,M)$ is a direct summand of a free module, and  projective.
Hence $(1) \Rightarrow (3)$.
\end{proof}

\begin{prop}\label{prop-F*-closed}
The following are equivalent:
\begin{enumerate}
\item $\bF^{*}$ is exact.
\item $H $ is a projective right $E$-module.
\item $M$ is $F_{*}$-closed.
\end{enumerate}
\end{prop}

\begin{proof}
(1)$\Leftrightarrow$(2): Since $E$ is Noetherian, $H_{E}$ is flat if and only if it is projective, which shows this equivalence.

(2)$\Leftrightarrow$(3): From the preceding lemma, $H $ is a projective right $E$-module
if and only if $_{R}M^{1/p} =F_{*} M \in \add(M) $, that is, $M$ is $F_{*}$-closed.
\end{proof}

\begin{prop}\label{prop-str-F*-closed-zero-kernel}
If $M$ is strongly $F_{*}$-closed, then for every nonzero $E$-module $A$, $\bF^{*}A$ is nonzero. 
Namely $\bF^{*}$ has zero kernel.
\end{prop}

\begin{proof}
If $M$ is strongly $F_{*}$-closed, then $H_{E}$ contains every indecomposable projective $E$-module
as a direct summand. Hence for some $l$, $H^{\oplus l}$ contains $E_{E}$ as a direct summand.
Hence $H^{\oplus l} \otimes_{E}A  \ne 0$ and $H \otimes_{E}A  \ne 0$.
\end{proof}

\begin{prop}\label{prop-Ftimes-closed->preserves-projectives}
Consider the following conditions:
\begin{enumerate}
\item  $\bF^{*}$ preserves projective modules.
\item  $\bF_{*}$ is exact.
\item  $H$ is a projective left $E^{1/p}$-module.
\item $M$ is $F^{\times}$-closed.
\end{enumerate}
Then $(1)\Leftrightarrow(2)\Leftrightarrow(3)\Leftarrow(4)$.
\end{prop}

\begin{proof}
(1)$\Leftrightarrow$(2)$\Leftrightarrow$(3): Obvious. 

(3)$\Leftarrow$(4): 
Let $F_{\rel}$ denote the relative Frobenius map, $R \hookrightarrow R^{1/p}$.
Define the reflexive pullback $F_{\rel}^{\times}$ similarly.
We have isomorphisms of left $E^{1/p}$-modules,
\[
H  \cong \Hom_{R^{1/p}}(F_{\rel}^{*}M  , M^{1/p})  
\cong \Hom_{R^{1/p}}(F_{\rel}^{\times}M  , M^{1/p}).
\]
Here the left isomorphism follows from the adjunction and the right from the fact that $M^{1/p}$ is reflexive.
Now (3)$\Leftarrow$(4) follows from Lemma \ref{lem-proj-right-left}.
\end{proof}

\begin{cor}\
If $M$ is strongly $F_{*}$-closed and $F^{\times}$-closed,
and if $\bF^{*}|_{\proj(E)}$ is  order-raising (regarded as an endofunctor), 
 then  $E$ has finite global dimension. Moreover if $E$ is CM, then 
its global dimension is $d$. 
\end{cor}

\begin{proof}
The functor $\bF^{*}$ preserves projectives and its restriction $\bF^{*}|_{\proj E}$
is exact and order-raising and has zero kernel. Now the corollary follows from 
Corollary \ref{cor-endofunctor-NCCR}.
\end{proof}

\subsection{The case where $M$ is a commutative regular local ring}

In this subsection, we additionally suppose that the $R$-module $M$ is also a commutative regular local ring 
such that the $R$-module structure on $M$ is the one induced from the ring extension $R \subset M$. 

\begin{lem}\label{lem-regular-strongly-closed}
$_{R}M$ is strongly $F_{*}$-closed and $\bF^{*}$ is exact and has zero kernel.
\end{lem}

\begin{proof}
Since $M$ is regular, 
$F_{*} M =M^{1/p}$ is  a free $M$-module. Hence $_{R}M$ is strongly $F_{*}$-closed. 
The rest assertions follow from Propositions \ref{prop-F*-closed} and \ref{prop-str-F*-closed-zero-kernel}.
\end{proof}

\begin{prop}\label{prop-CM-pure-Ftimes-closed}
Suppose that  $E = \End_{R}(M)$ is CM and that the ring extension $R \subset M$ is pure. Then  $M$ is $F^{\times}$-closed and $\bF^{*}$ preserves projectives.
\end{prop}

\begin{proof}
From Lemma \ref{lem-Hom-tensor-MCM}, $(M \otimes _{R}M )^{\refl}$ is also CM.
Since $M^{1/p}$ is a free $M$-module, $(M^{1/p} \otimes _{R}M )^{\refl}$ is also CM and 
a free $M^{1/p}$-module. Therefore $_{R^{1/p}}(M^{1/p} \otimes _{R}M )^{\refl}\in \add(_{R^{1/p}}M^{1/p})$.
Since $_{R^{1/p}}(R^{1/p} \otimes _{R}M )^{\refl}$ is a direct summand  of $_{R^{1/p}}(M^{1/p} \otimes _{R}M )^{\refl}$,
we have $_{R^{1/p}}(R^{1/p} \otimes _{R}M )^{\refl}\in \add(_{R^{1/p}}M^{1/p})$,
which saids that $M$ is $F^{\times}$-closed.
From Proposition \ref{prop-Ftimes-closed->preserves-projectives},
$\bF^{*}$ preserves projectives.
\end{proof}

Since $M$ is naturally regarded as a subring of $E = \End_{R}(M)$,
there are forgetting functors
$\mods (E) \to \mods(M)$ and similarly $\mods (E^{1/p})\to \mods(M^{1/p})$. 

\begin{prop}[\cite{Yasuda:2009wn}]\label{prop-frobenius-compatible}
The diagram
\[
\xymatrix{  
\mods(M)\ar[r] ^{F_{M}^{*}} & \mods(M^{1/p})  \\
\mods(E) \ar[r]_{\bF^{*}} \ar[u]& \mods(E^{1/p}) \ar[u]
}
\]
is commutative. Here $F_{M} : M \hookrightarrow M^{1/p}$ is the relative Frobenius map of $M$.
\end{prop} 

\begin{proof}
The two composite functors from $\mods(E)$ to $ \mods(M^{1/p})$ in the diagram 
send $_{E}E$  to $M^{1/p} \otimes_{M}E$ and $_{M^{1/p}}H$ respectively.
It suffices to show that there is a natural isomorphism between the two modules.
Since $M$ is regular and $M^{1/p}$ is a free $M$-module,
we have
\[
M^{1/p} \otimes_{M}E \cong \Hom_{R} (M,M^{1/p}) =H. 
\]
\end{proof}

\begin{prop}\label{prop-equiv-order-raising-radical}
Let $\fj$ be the Jacobson radical of $E$ and $\fm_{M}$ the maximal ideal of $M$.
Suppose that $\bF^{*}$ preserves projectives. Then $\bF^{*}|_{\proj(E)}$ is order-raising if and only if
$\fj \subset \fm_{M}E$.
\end{prop}

\begin{proof}
The ``if'' part:
Let $\phi: P \to Q$ be a morphism in $\proj(E)$ which factors through $\fj Q$, and hence through $\fm_{M}Q$.
From Proposition \ref{prop-frobenius-compatible},  $(\bF^{*})^{e} (\phi) $ factors through $\fm_{M}^{[p^{e}]} (\bF^{*})^{e} (Q)$. Here $\fm_{M}^{[p^{e}]}$ is the ideal of $M$ generated by $f^{p^{e}}$, $f \in \fm_{M}$.
Let $\fm_{R}$ be the maximal ideal of $R$. For every $i>0$, there exists $n>0$ such that 
$\fj^{i}\supset \fm_{R}^{n}E $. Then for every $n>0$, there exists $e>0$ such that
$\fm_{R}^{n}E \supset \fm_{M}^{[p^{e}]}E$. 
Therefore for every $i>0$, if $e$ is sufficiently large, then $(\bF^{*})^{e} (\phi) $ factors through $\fj^{i} (\bF^{*})^{e} (Q)$.
Thus $\bF^{*}|_{\proj(E)}$ is order-raising.

The ``only if'' part: Suppose $\fj \not \subset  \fm_{M}E$. Then choose an element 
$f \in \fj \setminus \fm_{M}E$
and let $\phi:{}_{E}E \to {}_{E}E$ be the map sending $1$ to $f$. From Proposition \ref{prop-frobenius-compatible},
for every $e$, $(\bF^{*})^{e}(\phi)$ does not factor through $\fm_{M}E$. 
However if $\bF^{*}$ were order-raising, for $e \gg 0$, $(\bF^{*})^{e}(\phi)$ would factor through $\fm_{M}E$. 
Therefore $\bF^{*}$ is not order-raising.
\end{proof}

We now prove  Corollary \ref{cor-pure-subring-regular-CMEnd->NCCR}
in a different way under the additional assumption on the Jacobson radical:

\begin{cor}
If $E$ is CM, the extension $R \subset M$ is pure
 and $\fj \subset \fm_{M}E$, then $E$ has global dimension $d$ and is a NCCR.
\end{cor}

\begin{proof}
From Proposition \ref{prop-CM-pure-Ftimes-closed}, 
$\bF^{*}$ preserves projectives.
From Lemma \ref{lem-regular-strongly-closed} and Proposition \ref{prop-equiv-order-raising-radical},
the restricted endofunctor $\bF^{*}|_{\proj(E)}$ is exact and order-raising and has zero kernel.
Now the assertion follows from Corollary \ref{cor-endofunctor-NCCR}.
\end{proof}

\begin{prob}
If $E$ is CM and the extension $R \subset M$ is pure, then is the Jacobson radical $\fj$
of $E$ included in  $\fm_{M}E$?
\end{prob}

%We can summarize results so far in the case where $M$ is a commutative and regular ring
% as in Figure \ref{fig-implications}.
%\begin{figure}
%\[
%\xymatrix{
%\text{\fbox{$R \subset M$ is pure and $E$ is CM}} \ar@{=>}[r]^{\text{Cor.\ \ref{cor-pure-subring-regular-CMEnd->NCCR}}} \ar@{=>}[d]_{\text{Prop.\ \ref{prop-CM-pure-Ftimes-closed}}}
%& \text{\fbox{$\gldim E=d$}}\ar@{=>}@/_/[d]_{+ \, \text{\fbox{$E$ is CM}} } \\
%\text{\fbox{$M$ is $F^{\times}$-closed}} \ar@{=>}[r] _{\text{Prop.\ \ref{prop-Ftimes-closed->preserves-projectives}}}  
%&\text{\fbox{$\bF^{*}$ preserves projectives}}
%\ar@{=>}@/_/[u] _{+  \, \text{\fbox{$\fj \subset \fm_{M}E$} $(\Leftrightarrow$\fbox{$\bF^{*}$ is order-raising})}} 
%}
%\]
%\caption{Implications in the case where $M$ is a commutative regular ring.}
%\label{fig-implications}
%\end{figure}

\subsection{Frobeniuses of skew group rings}

Let $S$ be a commutative regular local ring and $G$ a finite group acting on it.
Let 
\[
F_{S} :S \hookrightarrow S^{1/p}
\]
be the Frobenius map of $S$. Since $S$ is regular, $F_{S}$ is flat and $S^{1/p}$ is a free $S$-module.
We note that $S^{1/p}$ has a natural $G$-action such that $F_{S}$
is $G$-equivariant.   Therefore the skew group ring $S^{1/p}*G$ is also defined.

\begin{defn}\label{defn-Frobenius-for-skew}
We define the \emph{Frobenius map} of $S*G$ just as the inclusion map 
\[
\bF : S*G \hookrightarrow S^{1/p}*G,
\]
by which $S*G$ becomes a subring of $S^{1/p}*G$.
Accordingly we define the \emph{Frobenius pullback and pushforward functors}
\begin{align*}
\bF^{*}:\mods( S*G) \to \mods (S^{1/p}*G), \ A \mapsto S^{1/p}*G \otimes_{S*G}A ,  \\
\bF_{*} :  \mods (S^{1/p}*G )\to \mods (S * G ) , \ {}_{S^{1/p}*G}A \mapsto {}_{S*G}A.
\end{align*}
\end{defn}

The following proposition is a direct consequence of the definition:

\begin{prop}\label{skew-Frob-exact-preserve-proj-zero-kernel}
$\bF^{*}$ is exact, preserves projectives and has zero kernel.
Also $\bF_{*}$ is exact. Furthermore the diagram
\[
\xymatrix{  
\mods(S)\ar[r] ^{F^{*}_{S}} & \mods(S^{1/p})  \\
\mods(S*G) \ar[r]_{\bF^{*}} \ar[u]& \mods(S^{1/p}*G) \ar[u]
}
\]
is (2-)commutative.
\end{prop}

\begin{cor}\label{cor-skew-seven-equivalent}
Let $\fn$ be the maximal ideal of $S$. 
Suppose that $G \subset \Aut (S)$, that
$G$ acts trivially on the residue field $S/\fn$ and that $S$ is unramified over $R=S^{G}$ in codimension one.
Then
the following are equivalent:
\begin{enumerate}
\item The five equivalent conditions in Corollary \ref{cor-byproduct} hold.
\item The Jacobson radical of $S*G$ is $ \fn *G $.
\item $\bF^{*}$ is order-raising.
\end{enumerate}
\end{cor}

\begin{proof}
(1)$\Rightarrow$(2): Villamayor's theorem \cite{Villamayor:1958tc} (see also \cite{Passman:1983ue}).

(2)$\Rightarrow$(3): Similar to Proposition \ref{prop-equiv-order-raising-radical}.

(3)$\Rightarrow$(1): This follows from Proposition \ref{skew-Frob-exact-preserve-proj-zero-kernel} and
Cororally \ref{cor-endofunctor-NCCR}.
\end{proof}

\subsection{Agreement of the two definitions of noncommutative Frobenius}

\begin{prop}\label{prop-Frobs-coincide}
Let $S$ be a commutative regular local ring and $G \subset \Aut(S)$  a finite group 
of automorphisms of $S$.
Suppose that $S$ is unramified over $R=S^{G}$ in codimension one.
Then with identifications $S*G = \End_{R}(S)$ and $S^{1/p}*G=\End_{R^{1/p}}(S^{1/p})$,
the Frobenius morphisms of $S*G$ and $\End_{R}(S)$
in Definitions \ref{defn-Frobenius-for-end} and \ref{defn-Frobenius-for-skew} coincide.
\end{prop}

\begin{proof}
Frobenius morphisms are respectively given by  the bimodules 
\[
_{S^{1/p}*G}(S^{1/p}*G)_{S*G} \text{ and }
_{\End_{R^{1/p}}(S^{1/p})}\Hom_{R}(S,S^{1/p})_{\End_{R}(S)} . 
\]
Therefore it suffices to show the isomorphism of these bimodules
under the mentioned identifications of rings.
By assumption, we have an isomorphism of $(S,R^{1/p})$-bimodules, 
$S^{1/p} \cong (S \otimes _{R} R^{1/p})^{\refl}$. 
Therefore we have isomorphisms of bimodules,
\begin{align*}
S^{1/p}*G& \cong \Hom_{R^{1/p}}(S^{1/p},S^{1/p}) \\
&\cong \Hom_{R^{1/p}}( (S \otimes_{R}R^{1/p})^{\refl} ,S^{1/p}) \\
& \cong \Hom_{R^{1/p}}( S \otimes_{R}R^{1/p},S^{1/p}) \\
&\cong \Hom_{R}( S,S^{1/p})
\end{align*} 
We have completed the proof.
\end{proof}

\end{document}